\documentclass[11pt,a4paper]{article}
\usepackage{epsf,epsfig,amsfonts,amsgen,amsmath,amstext,amsbsy,amsopn,amsthm,lineno, caption}
\usepackage{color}
\usepackage{graphicx}
\usepackage{subfigure}
\usepackage{epstopdf}
\usepackage{authblk}

\newtheorem{theorem}{Theorem}

\newtheorem{lemma}{Lemma}
\newtheorem{cor}{Corollary}

\theoremstyle{definition}
\newtheorem{exam}{Example}
\newtheorem{definition}{Definition}
\newtheorem{claim}{Claim}

\newtheorem{conjecture}{Conjecture}

\begin{document}
	\title
	{\Large Properly colored cycles in edge-colored complete graphs containing no monochromatic triangles: a vertex-pancyclic analogous result 
	}
	
	\date{}
	\author[1,2]{\small Ruonan Li \thanks{E-mail: rnli@nwpu.edu.cn}}
	\affil[1]{School of Mathematics and Statistics, Northwestern Polytechnical University\\
		
		Xi'an, 710129, P.R.~China}
	\affil[2]{Xi'an-Budapest Joint Research Center for Combinatorics,
		Northwestern Polytechnical University\\
		
		 Xi'an, 710129, P.R.~China}
	\maketitle
	\abstract
	
	A properly colored cycle (path) in an edge-colored graph is a cycle (path) with consecutive edges assigned distinct colors. A monochromatic triangle is a cycle of length $3$ with the edges assigned a same color. It is known that every edge-colored complete graph without containing monochromatic triangles always contains a properly colored Hamilton path. In this paper, we investigate the existence of properly colored cycles in edge-colored complete graphs when monochromatic triangles are forbidden. We obtain a vertex-pancyclic analogous result combined with a characterization of all the exceptions.
	
	%
	
	\noindent
	\section{Introduction} 
	Let $G$ be an undirected graph. An {\it edge-coloring} of $G$ is a mapping $col: E(G)\rightarrow \mathbb{N}$, where $\mathbb{N}$ is the natural number set. A graph $G$ is called an {\it edge-colored graph} if $G$ is assigned an edge-coloring. Denote by $col(e)$ and $col(G)$, respectively, the color of an edge $e$ and the set of colors assigned to $E(G)$. An edge-colored graph is called {\it properly colored} (or {\it PC}  for short) if adjacent edges are assigned distinct colors, is called {\it rainbow} if $|col(G)|=|E(G)|$, and is called {\it monochromatic} if $|col(G)|=1$.
	For a vertex $v \in V(G)$, the {\it color degree} of $v$ in $G$, denoted by $d^c_G(v)$ is the number of distinct colors assigned to the edges incident to $v$. We use $\delta^c(G)=\min\{d^c_G(v):v\in V(G)\}$ to denote the {\it minimum color degree} of $G$. For a color $i\in col(G)$, let $G^i$ be the subgraph of $G$ induced by all the edges of color $i$. We use $\Delta^{mon}(G)=\max\{\Delta(G^i):i\in col(G)\}$ to denote the {\it maximum monochromatic degree} of $G$.  For two disjoint subsets $V_1$ and $V_2$ of $V(G)$, denote by $col(V_1,V_2)$ the set of colors appearing on the edges between $V_1$ and $V_2$ in $G$. When $V_1=\{v\}$, use $col(v,V_2)$ to denote $col(\{v\},V_2)$. Represent by {\it triangle}, {\it quadrangle} and {\it pentagon} the cycles of lengths $3,4$ and $5$, respectively. For other notation and terminology not defined here, we refer the reader to \cite{Bondy:2008}.
	
	This paper mainly investigates the existence of PC cycles in edge-colored complete graphs when monochromatic triangles are forbidden. Barr \cite{Barr:1998} proved that an edge-colored $K_n$ without containing monochromatic triangles must contain a PC Hamilton path. This can be easily verified by assuming a longest PC path, and immediately extending it when it is not a Hamilton path. Therefore we wonder whether the condition ``no monochromatic triangles'' can also in a large probability make an edge-colored $K_n$ to contain a PC Hamilton cycle. 
	
	To grasp some intuition, we first establish two examples which show that an edge-colored $K_n$ may neither contain a monochromatic triangle nor a PC Hamilton cycle.
	
	\begin{exam}\label{exam:K5}
		Let $G$ be an edge-colored $K_5$ with $V(G)=\{v_1,v_2,v_3,v_4,v_5\}$. The cycle $v_1v_2v_3v_4v_5v_1$ is of color $1$ and the cycle $v_1v_3v_5v_2v_4v_1$ is of color $2$.
	\end{exam}
	\begin{exam}\label{exam:direct}
		Let $G$ be an edge-colored $K_n$ containing no monochromatic triangles and $v_1,v_2,v_3$ are three distinct vertices in $G$ with $col(v_1v_2)=1$, $col(v_2v_3)=2$, $col(v_3v_1)=3$ and $col(v_i, V(G)\setminus\{v_1,v_2,v_3\})=\{i\}$ for all $i=1,2,3$. 
	\end{exam}
	
	In Example \ref{exam:K5}, since $|col(G)|=2$, no PC pentagon exists.
	In Example \ref{exam:direct}, let $V=\{v_1,v_2,v_3\}$. 
	For each edge $e=uv_i$ with $u$ in $G-V$, by the coloring, from the vertex $v_i$, we can only extend $uv_i$ within $V$ (to obtain a longer PC path). Thus there is no PC cycle in $G$ containing edges between $V$ and $G-V$. Hence no PC Hamilton cycle exists. Generally, we give the following definition, which was firstly proposed in  \cite{RN-BL-SG:2020}. 
	\begin{definition}\label{def:non-deg}
		Let $G$ be an edge-colored graph. If there exists a nonempty set $S\subseteq V(G)$ and a function $f: S\rightarrow \mathbb{N}$ such that for each edge $e$ joining $u$ and $v$, the following holds: \\
		(i)  if $u,v\in S$, then $col(e)=f(u)$ or $col(e)=f(v)$;\\
		(ii) if $u\in S$ and $v\in V(G)\setminus S$, then $col(e)=f(u)$,\\
		then we say $G$ is {\it semi-degenerate}, $S$ is a {\it degenerate set} of $G$, and say $f$ is a compatible to $S$ in $G$ . In particular, if $S=V(G)$, then we say $G$ is {\it degenerate} and $f$ is compatible to $G$ ; 
		if $S$ does not exist, then we say $G$ is {\it non-degenerate}.
	\end{definition}
	Obviously, in Example \ref{exam:direct}, $\{v_1,v_2,v_3\}$ is a degenerate set of $G$. We can further conclude that if an edge-colored graph $G$ contains a degenerate set $S\neq V(G)$, then each edge between $S$ and $G-S$ is not contained in any PC cycles. Hence no PC Hamilton cycle exists.
	
	Interestingly, we find that these two kinds of examples above are the only counterexamples for the existence of PC Hamilton cycles in edge-colored $K_n$ containing no monochromatic triangles. In fact, a stronger theorem is obtained as following, which can be regarded as a vertex-pancyclic analogous result.
	
	\begin{theorem}\label{thm:main}
		Let $G$ be an edge-colored $K_n$ ($n\geq 4$) containing no monochromatic triangles. Then one of the following holds:\\
		$(a)$ each vertex of $G$ is contained in PC cycles of all lengths from $4$ to $n$;\\
		$(b)$ $G$ has a degenerate set $S\neq V(G)$.\\
		$(c)$ $G$ is an edge-colored $K_5$ composed of two edge-disjoint monochromatic pentagons.
	\end{theorem}
    We postpone the proof of Theorem \ref{thm:main} to Section \ref{sec:3}. In the next section, we give several additional results and analyses related to Theorem \ref{thm:main}.
 \section{Additional results and analyses}  
	The absence of PC triangles in the statement of Theorem \ref{thm:main}(a) can be explained by the following well-known structural theorem.
\begin{theorem}[Gallai \cite{Gallai:1967}]\label{thm:Gallai}
	Let $G$ be an edge-colored complete graph containing no PC triangles. Then $V(G)$ can be partitioned into $V_1,V_2,\ldots,V_k$ such that $|\cup_{1\leq i<j\leq p}col(V_i,V_j)|\leq 2$ and  $|col(V_i,V_j)|=1$ for all $i,j$ with  $1\leq i<j\leq p$.
\end{theorem}
From Theorem \ref{thm:Gallai}, we can easily construct infinite number of  non-degenerate edge-colored complete graphs with each triangle assigned exactly two colors, namely no monochromatic nor PC triangles exist. 

	The pancyclic  properties in graphs and digraphs have been well studied many decades before. See \cite{Bondy: 1971}\cite{Moon:1966} and \cite{Schmeichel-Hakimi:1974}. 
	In edge-colored complete graphs, Fujita and Magnant \cite{Fujita:2011} conjectured the following.
	\begin{conjecture}[Fujita and Magnant \cite{Fujita:2011}]\label{con:Fujita}
		Let $G$ be an edge-colored $K_n$ with $\delta^c(G)\geq \frac{n+1}{2}$. Then each vertex of $G$ is contained in PC cycles of all lengths from $3$ to $n$.
	\end{conjecture}
	Under the condition of Conjecture \ref{con:Fujita}, Fujita and Magnant \cite{Fujita:2011} proved  that each vertex is contained in a PC triangle and a PC quadrangle; Li et al.\cite{R.Li:2017} showed that each vertex of $G$ is contained in PC cycles of length at least $\delta^c(G)$; Chen et al. \cite{Chen:2019} confirmed the conjecture when no monochromatic triangles exist.
	
	\begin{theorem}[Chen et al. \cite{Chen:2019}]\label{thm:Chen}
		Let $G$ be an edge-colored $K_n$ with $\delta^c(G)\geq \frac{n+1}{2}$. If $G$ contains no monochromatic triangles, then each vertex of $G$ is contained in PC cycles of all lengths from $3$ to $n$.
	\end{theorem}
	
	Theorem \ref{thm:main} implies that the condition `` no monochromatic triangles '' is quite strong since we can almost obtain the `` vertex-pancyclic '' result without using the color degree condition, which is required in Conjecture \ref{con:Fujita} and Theorem \ref{thm:Chen}. 
	When Theorem \ref{thm:main}$(b)$ or $(c)$ happens, it is easy to verify that the minimum color degree of $G$ must be smaller than $(n+1)/2$. Combining Theorem \ref{thm:main} with the result by Fujita and Magnant \cite{Fujita:2011} that each vertex is contained in a PC triangle, we can confirm  Conjecture \ref{con:Fujita} in the case that no monochromatic triangle exists, namely, obtain Theorem \ref{thm:Chen} as a corollary.
	
%
%
%
	
	The other conjecture on the existence of PC Hamilton cycles was given by Bollob\'{a}s and Erd\H{o}s \cite{Bollobas-Erdos:1976}.
	
	\begin{conjecture}[Bollob\'{a}s and Erd\H{o}s \cite{Bollobas-Erdos:1976}]
		Let $G$ be an edge-colored $K_n$. If $\Delta^{mon}(G)<\lfloor\frac{n}{2}\rfloor$, then $G$ contains a PC Hamilton cycle.
	\end{conjecture}
	Lo \cite{Lo:2016+} confirmed this conjecture asymptotically. When Theorem \ref{thm:main}$(b)$ or $(c)$ happens, it is also easy to verify that $\Delta^{mon}(G)\geq \lfloor\frac{n}{2}\rfloor$. So we can obtain the following corollary.
	\begin{cor}\label{cor:mono}
		Let $G$ be an edge-colored $K_n$. If $\Delta^{mon}(G)<\lfloor\frac{n}{2}\rfloor$ and $G$ contains no monochromatic triangles, then each vertex of $G$ is contained in PC cycles of all lengths from $4$ to $n$. 
	\end{cor}
    The absence of PC triangles in Corollary \ref{cor:mono} is reasonable. Since there exist an infinite number of edge-colored $K_n$ with $\Delta^{mon}(K_n)=2n/5$ and containing no PC triangles (see \cite{Gyarfas: 2004} for the construction). 
    
	For general edge-colored graphs (not necessarily complete), Lo \cite{Lo:2014-2} proved the following: for any $\varepsilon>0$, there exists an integer $n_{0}=n_{0}(\varepsilon)$ such that every edge-colored graph $G$ with $n\geq n_{0}$ and $\delta^{c}(G)\geq(\frac{2}{3}+\varepsilon)|G|$ contains PC cycles of all lengths from $3$ to $|G|$. For other results related to PC Hamilton cycles, we refer the reader to \cite{Lo:2016+, Lo:2019}.	
	
	\section{Preliminaries}
	To deliver the proof of Theorem 1, we need some auxiliary lemmas.	First two lemmas are concerning directed cycles in tournaments and multipartite tournaments, which will be useful when dealing with PC cycles in a degenerate edge-colored $K_n$. Given a digraph $D$. If $uv$ is an arc of $D$, then we say $u$ {\it dominates} $v$, and denote it by $u\rightarrow v$. For two disjoint subsets $A,B$ of $V(D)$, if each arc $uv$ between $A$ and $B$ satisfies that $u\in A$ and $v\in B$, then we say $A\rightarrow B$. For $H\subset D$ and $S\subseteq V(D)\setminus V(H)$, we use $H\rightarrow S$ ($S\rightarrow H$) to denote $V(H)\rightarrow S$ ($S\rightarrow V(H)$). If $|S|=1$, say $S=\{v\}$, then we use $H\rightarrow v$ ($v\rightarrow H$) to denote $V(H)\rightarrow \{v\}$ ($\{v\}\rightarrow V(H)$). For a vertex $u\in V(D)$, denote by $N^+(u)$ the set of vertices that are dominated by $u$, and denote by $N^-(u)$ the set of vertices that are dominating $u$. For other terminology and notation on digraphs, we refer the reader to \cite{Bang-Gutin:2008}.

	\begin{lemma}[Moon \cite{Moon:1966}]\label{lem:Moon}
		Each vertex in a strongly connected tournament of order $n$ is contained in directed cycles of all lengths from $3$ to $n$.
	\end{lemma}
	\begin{lemma}\label{lem:directed}
		For $t\geq 1$, let $T$ be a strongly connected $k$-partite tournament with partite sets $V_1,V_2,\ldots,V_k$ with $V_i=\{x_i,y_i\}$, $V_{j}=\{v_j\}$ for all $i\in [1,t]$ and $j\in [t+1,k]$. If $|V(T)|\geq 4$ and $N^+(x_i)\cap N^+(y_i)=\emptyset$ for all $i\in [1,t]$, then each vertex of $T$ is contained in directed cycles of all lengths from $4$ to $|V(T)|$.
	\end{lemma}
	\begin{proof}
		
		We firstly prove that each vertex of $T$ is contained in a directed quadrangle. 	
		
		If $t=1$, then note that $T$ is strongly connected. We assume that $x_1\rightarrow v_i$ and $y_1\rightarrow v_j$. Since $N^+(x_1)\cap N^+(y_1)=\emptyset$, we have $v_i\neq v_j$ and $x_1v_iy_2v_jx_i$ is a directed quadrangle containing $\{x_1, y_1\}$.
		If $t\geq 2$, then consider the bipartite tournament $B$ induced by $(V_1,V_2)$. Since $N^+(x_i)\cap N^+(y_i)=\emptyset$ for all $i\in [1,t]$, we have $d^+_B(v)\geq 1$ for each $v\in V(B)$. Thus $B$ contains a directed cycle, which must be a directed quadrangle containing $\{x_1, y_1\}$. 	
		By the symmetry between $\{x_1,y_1\}$ and $\{x_i,y_i\}$ $(i\in [1,t])$, we can see that each vertex in $\cup_{i=1}^tV_i$ is contained in a directed quadrangle. 
		
		Now we prove that each $v_j$ ($t+1\leq j\leq k$) is also contained in a directed quadrangle. By the above analysis, we can assume that $x_1ay_1bx_1$ is a directed quadrangle containing $\{x_1,y_1\}$, and $\{a,b\}=\{x_2,y_2\}$ when $t\geq 2$. For each $j\in[t+1,k]$, since $N^+(x_1)\cap N^+(y_1)=\emptyset$, either $v_j\rightarrow x_1$ or $v_j\rightarrow  y_1$. By the symmetry between $x_1$ and $y_1$, assume that $v_j\rightarrow x_1$. 
		If  $y_1\rightarrow  v_j$, then either $v_j\in \{a,b\}$, which leads $v_j$ lying in a directed quadrangle,  or $v_j\not\in \{a,b\}$, which implies $v_jx_1ay_1v_j$ is a directed quadrangle. So we assume that  $v_j\rightarrow x_1$ and $v_j\rightarrow y_1$.
\begin{itemize}
	\item  If $\{a,b\}\cap N^+(v_j)\neq \emptyset$ and $\{a,b\}\cap N^-(v_j)\neq \emptyset$, then either $v_jbx_1av_j$ or $v_jay_1bv_j$ is a directed quadrangle.
	\item  If $\{a,b\}\subseteq N^-(v_j)$, then $v_j\in N^+(a)\cap N^+(b)$, which implies $t=1$. Consider the arc between $a$ and $b$. Then either $v_jy_1bav_j$ or $v_jx_1abv_j$ is a directed quadrangle. 
	
\end{itemize}
		By the symmetry between $\{x_1,y_1\}$ and $\{x_i,y_i\}$ $(i\in [1,t])$, the only case left is that $v_j\rightarrow \cup_{i=1}^tV_i\cup\{a,b\}$. Define $ U=\cup_{i=1}^tV_i\cup\{a,b\}.$	
		Since $T$ is strongly connected, there exists a directed path from $U$ to $v_j$. Let $P=z_0z_1z_2\cdots z_sv_j$ be a directed path with $z_0\in U$ and $z_i\in \{v_{t+1}, v_{t+2}, \ldots, v_k\}$ for all $i\in [1,s]$. Since there exists a directed quadrangle in $T[U]$ containing $z_0$, we can find a vertex $x\in U$ such that $x\rightarrow z_0$. Let $P'=xP$. Note that $v_j\rightarrow U$. So $s\geq 1$ and $v_j\rightarrow x$. Thus $|P'|\geq 4$ and $T[V(P')]$ is a strongly connected tournament of order at least $4$. By Lemma \ref{lem:Moon}, $v_j$ is contained in a directed quadrangle.
		
		Let $u^*$ be an arbitrary vertex in $T$ and $\ell$ an arbitrary integer in $[4,|V(T)|-1]$. To complete the proof, we need to show that if there exists a directed cycle of length $\ell$ containing $u^*$, then we can obtain a directed cycle of length $\ell+1$ containing $u^*$. 
		
		By contradiction. Suppose that $C=u_1u_2\cdots u_{\ell}u_1$  is a PC cycle containing $u^*$, but $u^*$ is not contained in any PC cycle of length $\ell+1$. Let $S=V(T)\setminus V(C)$. 
For two distinct vertices $u_i,u_j\in V(C)$, use $u_iCu_j$ to denote the directed path on $C$ from $u_i$ to $u_j$. 
		Define
		$$S_1=\{x\in S: \exists y\in V(C)~s.t.~x \text{~is not adjacent to~} y\},$$
		$$S_2=\{x\in S\setminus S_1: C\rightarrow x\},$$
		$$S_3=\{x\in S\setminus S_1:  x\rightarrow C\},$$
		$$S_0=S\setminus \cup_{i=1}^3S_i.$$
		\begin{claim} \label{cl:lem1}
		$C\rightarrow S_1$ when $S_1\neq \emptyset$.
		\end{claim}
	\begin{proof}
		Let $x$ be an arbitrary vertex in $S_1$. By the definition of $S_1$, there is a vertex $y\in V(C)$, say $y=u_1$, such that $x$ and $y$ come from a same partite set. Then $x$ is incident to all the vertices on $C-y$. Note that $N^+(x)\cap N^+(y)=\emptyset$ and $y\rightarrow u_2$. So $u_2\rightarrow x$. Since $xu_3Cyu_2x$ is not a directed cycle, we have $u_3\rightarrow x$. Repeat this process. We finally get $C\rightarrow x$. Hence $C\rightarrow S_1$  when $S_1\neq \emptyset$.
	\end{proof}
	\begin{claim}
		$S_3\rightarrow S_1\cup S_2$.
	\end{claim}	
    \begin{proof}
	    Suppose to the contrary that there exists a vertex $z\in S_3$ and $x\in S_1\cup S_2$ such that $x\rightarrow z$. Since $\ell \geq 4$, there exists a vertex $u_i\in V(C)$ such that $u_i\neq u^*$ and $u_{i-1}$ is adjacent to $x$. By Claim \ref{cl:lem1} and the definitions of $S_2$ and $S_3$, we can see that $u_{i-1}xzu_{i+1}Cu_{i-1}$ is a directed cycle of length $\ell+1$ containing $u^*$, a contradiction. 
    \end{proof}

		If $S_0=\emptyset$, then $S=S_1\cup S_2\cup S_3$. Since $S\neq \emptyset$, either $S_1\cup S_2\neq \emptyset$ or $S_3\neq\emptyset$. Therefore either $V(C)\cup S_3\rightarrow S_1\cup S_2$ or $S_3\rightarrow V(C)\cup S_1\cup S_2$. In both cases, we can see that $T$ is not strongly connected, a contradiction.
		
		If $S_0\neq \emptyset$, then let $v\in S_0$. We can see that $v$ is incident to all the vertices of $C$ with $N^+(v)\cap V(C)\neq \emptyset$ and $N^-(v)\cap V(C)\neq \emptyset$. Hence there must exist vertices $u_i\in N^-(v)$ and $u_{i+1}\in N^+(v)$ (indices are taken modulo $\ell$). So $u_ivu_{i+1}Cu_i$ a directed cycle of length $\ell+1$ containing $u^*$, our final contradiction.
	\end{proof}
	Next two lemmas deal with PC cycles in edge-colored $K_n$. Let $C$ be a cycle in an undirected graph $G$. Give $C$ a direction (either clockwise or  anti-clockwise). For a vertex $u\in V(C)$, use $u^+$ and $u^-$, respectively, to denote the successor and predecessor of $u$ on $C$ along the direction of $C$. Let $u^{++}=(u^+)^+$ and $u^{--}=(u^-)^-$. For $u,v\in V(C)$, we use $uCv$ to denote the segment between $u$ and $v$ along the direction of $C$ and use $u\bar{C}v$ to denote the the segment between $u$ and $v$ along the inverse direction of $C$.
		\begin{lemma}\label{lem:not ext}
		Let $G$ be an edge-colored $K_n$ containing no monochromatic triangles. Let $C$ be a PC cycle in $G$ with a given direction and $v$ a vertex in $G-C$. If there is no PC cycle $C'$ satisfying $ V(C')=V(C)\cup \{v\}$. Then one of the following statements holds:
		($a$) $|col(v,C)|=1$;
		($b$)  $col(vu)=col(uu^-)$ for each vertex $u\in V(C)$;
		($c$) $col(vu)=col(uu^+)$ for each vertex $u\in V(C)$.
	\end{lemma}
	\begin{proof}
		If $|col(v,C)|\geq 2$, then there is a vertex $x\in V(C)$ such that $col(vx)\neq col(vx^+)$. Since $vx^+Cxv$ is not a PC cycle, either $col(vx)=col(xx^-)$ or $col(vx^+)=col(x^+x^{++})$. Without lose of generality, assume that $col(vx)=col(xx^-)$. This forces $col(vx^{-})=col(x^-x^{--})$ (otherwise, we either obtain a monochromatic triangle or a PC cycle on $V(C)\cup\{v\}$). Now  $col(vx^{-})\neq col(vx)$ and  $col(vx^{-})=col(x^-x^{--})$. Repeat this process. We can finally get $col(vu)=col(uu^-)$ for each vertex $u\in V(C)$.
	\end{proof}
	\setcounter{claim}{0}
	\begin{lemma}\label{lem:C4}
		Let $G$ be a non-degenerate edge-colored $K_n~(n\geq 4)$ containing no monochromatic triangles. Then each vertex of $G$ is contained in a PC quadrangle.
	\end{lemma}
	\begin{proof}
		By contradiction. Suppose that $v$ is a vertex in $G$ not contained in any PC quadrangle. Let $\{1,2,\ldots,k\}$ be the set of colors appearing on the edges incident to $v$. Since $G$ is non-degenerate, we have $k\geq 2$. Define $V_i=\{u\in V(G): col(uv)=i\}$ for $i=1,2,\ldots,k$ and assume that $|V_1|\geq |V_2|\geq \cdots \geq |V_k|$. 
		\begin{claim}\label{cl:C4}
			For distinct vertices $x,y,z,w$ in $G-v$, if $col(v, \{x,y\}) \cap col(v, \{z,w\})=\emptyset$, then either $col(xy)\in col(v, \{x,y\})$ or $col(zw)\in col(v,\{z,w\})$.
		\end{claim}
		\begin{proof}
			Suppose to the contrary that $col(xy)\not\in col(v, \{x,y\})$ and $col(zw)\not\in col(v,\{z,w\})$. Let $col(xy)=\alpha$, $col(zw)=\beta$. Since $col(v,\{x,y\})\cap col(v, \{z,w\})=\emptyset$, we can assume that $col(vx)=c_1$,  $col(vz)=c_2$ and $c_1\neq c_2$. Note that $vzxyv$ and $vxzwv$ are not PC quadrangles, we have $col(xz)\in  \{\alpha,c_2\}\cap\{c_1,\beta\}$. Recall that $\alpha\neq c_1$ and $\beta\neq c_2$. We get $col(xz)=\alpha=\beta$. Similarly, we can obtain $col(yz)=\alpha=\beta$, which forces $xyzx$ being a monochromatic triangle, a contradiction.
		\end{proof}
		By Claim \ref{cl:C4} and the fact that $G$ contains no monochromatic triangles, we get $|V_2|=1$. Therefore $|V_i|=1$ for all $i\in [2,k]$. Let $v_i$ be the unique vertex in $V_i$ for $i\in [2,k]$. 
		
		If $|V_1|=1$, then $k=n-1$. Let $v_1$ be the unique vertex in $V_1$. Since $G$ is non-degenerate, there are distinct vertices, say $v_1$ and $v_2$, satisfying $col(v_1v_2)\not\in\{1,2\}$. Let
		$col(v_1v_2)=\alpha$. Since $n\geq 4$, the vertex $v_3$ exists. If $col(v_2v_3)=\alpha$, then consider cycles $vv_1v_3v_2v$ and $vv_2v_1v_3v$. We get $col(v_1v_3)\in \{1,\alpha\}\cap\{3,\alpha\}=\{\alpha\}$. 
		Thus $v_1v_2v_3v_1$ is a monochromatic triangle, a contradiction. So $col(v_2v_3)\neq \alpha$. Consider the cycle $vv_1v_2v_3v$. We get $col(v_2v_3)=3$. Note the symmetry between $v_1$ and $v_2$. We can also get $col(v_1v_3)=3$. By similar arguments, we can get $col(v_jv_1)=col(v_jv_2)=j$ for all $j\in[3,k]$.
		For each pair of distinct vertices $v_i$ and $v_j$ with $i,j\geq 3$, apply Claim \ref{cl:C4} to $(v_1,v_2,v_i,v_j)$, we get $col(v_iv_j)\in\{i,j\}$. Thus $\{v_3,v_4,\ldots,v_k\}$ is a degenerate set of $G$, a contradiction.
		
		If $|V_1|=2$, then $k=n-2$. Let $V_1=\{x,y\}$ and assume that $col(xy)=\alpha$. Since $G$ has no monochromatic triangles, we have $\alpha\neq 1$. For each pair of distinct vertices $v_i$ and $v_j$ with $i,j\geq 2$, apply Claim \ref{cl:C4} to $(x,y,v_i, v_j)$, there holds $col(v_iv_j)\in\{i,j\}$. Consider the cycle $xyv_ivx$, we get $col(yv_i)=i\text{~or~}\alpha$. Note the symmetry between $x$ and $y$. We also get $col(xv_i)=i\text{~or~}\alpha$. Let $f(v_i)=i$ for all $i\in [2,k]$, let $f(v)=1$ and $f(x)=f(y)=\alpha$. We can see that $G$ is degenerate, a contradiction.
		
		If $|V_1|\geq 3$, then there exist vertices $x,y,z\in V_1$ such that $col(xy)\neq col(yz)$ (since $G[V_1]$ contains no monochromatic triangles). Let $col(xy)=\alpha$ and $col(yz)=\beta$. Then $1\not\in \{\alpha,\beta\}$. For each pair of distinct vertices $v_i$ and $v_j$, apply Claim \ref{cl:C4} to $(x,y,v_i,v_j)$, we get $col(v_iv_j)\in\{i,j\}$. Consider cycles $vv_iyxv$ and $vv_iyzv$. We get $col(yv_i)=i$ for all $i\in [2,k]$.
		Let $f(v_i)=i$ for all $i\in [2,k]$ and let $f(v)=1$. Let $H$ be the edge-colored subgraph of $G$ induced by $V_1$.  For each vertex $u\in V(H)$ with $d^c_H(u)\geq 2$, by a similar argument applied to $y$, we get $col(uv_i)=i$ for all $i\in[2,k]$.  
		For each vertex $w\in V(H)$ with $d^c_H(w)=1$, obviously, $w\neq y$. Let $f(w)$ be the unique color in $col(w, V_1-w)$. Note that $col(wy)\neq 1$ (otherwise $wyvw$ is a monochromatic triangle) and consider the cycle $wyvv_iw$. We get $col(wv_i)=i$ or $f(w)$. Let $S=\{w\in V(H): d^c_H(w)=1\}$. Then it is easy to see that $\cup_{i=2}^k V_i\cup S\cup \{v\}$ is a degenerate set of $G$, a contradiction.
		
		The proof is complete. 
	\end{proof}
	\setcounter{claim}{0}
	\section{Proof of Theorem \ref{thm:main}}\label{sec:3}
	When Theorem \ref{thm:main}(b) fails, either $G$ is degenerate or  non-degenerate. In this section, we prove Theorem \ref{thm:main} under these two cases, which are delivered in Theorems \ref{thm:deg-vertex-pan} and  \ref{thm:non-deg} respectively.
	\begin{theorem}\label{thm:deg-vertex-pan}
		Let $G$ be a degenerate edge-colored $K_n~(n\geq 4)$ without containing monochromatic triangles. If each set $S\subset V(G)$ is not a degenerate set of $G$, then each vertex of $G$ is contained in PC cycles of all lengths from $4$ to $n$.
	\end{theorem}
	\begin{proof}
		Since $G$ is degenerate, there exists a compatible function $f: V(G)\rightarrow col(G)$. Assume that $col(G)=\{1,2,\ldots,k\}$ and $V_i=\{u\in V(G): f(u)=i\}$ for all $i\in[1,k]$. Then $|V_i|\leq 2$ (otherwise $G[V_i]$ contains a monochromatic triangle). Define a directed graph $D$ with $V(D)=V(G)$ and 
		$$A(D)=\{uv: col(uv)=f(u)\text{~and~} col(uv)\neq f(v)\}.$$
		Since $G$ is degenerate and containing no degenerate set $S\neq V(G)$, $D$ is actually a strongly connected $k$-partite tournament with partite sets $V_1,V_2,\ldots, V_k$. Note that $G$ contains no monochromatic triangles, we have $N^+(x)\cap N^+(y)=\emptyset$ for each pair of distinct vertices $x$ and $y$ from a same $V_i$. Apply Lemma \ref{lem:directed} to $D$. Then each vertex of $D$ is contained in directed cycles of all lengths from $4$ to $n$, which implies that each vertex of $G$ is contained in PC cycles of all  lengths from $4$ to $n$.
	\end{proof}
	\begin{theorem}\label{thm:non-deg}
		Let $G$ be a non-degenerate edge-colored $K_n~(n\geq 4)$ without containing monochromatic triangles. Then each vertex of $G$ is contained in PC cycles of all lengths from $4$ to $n$, unless $G$ is a $K_5$ containing two edge-disjoint monochromatic pentagons.
	\end{theorem}
	\begin{proof}
		Let $G$ be a non-degenerate edge-colored complete graph without containing monochromatic triangles. By Lemma \ref{lem:C4}, each vertex of $G$ is contained in a PC quadrangle. If $n=4$, then the proof is done. Now assume that $n\geq 5$ and $u^*$ is a vertex in $G$ contained in a PC cycle $C$ of length $\ell\in [4,n-1]$ but not in any PC cycle of length $\ell+1$. We will show that $G$ is a $K_5$ containing two edge-disjoint monochromatic pentagons.
		
		Give $C$ a direction. Let $S=V(G)\setminus V(C)$,  
		$$S_1=\{v\in S: |col(v,C)|=1\},$$ $$S_2=\{v\in S\setminus S_1: col(vu)=col(uu^+) \text{~for all~} u \in V(C)\}$$ and 
		$$S_3=\{v\in S\setminus S_1: col(vu)=col(uu^-) \text{~for all~} u \in V(C)\}.$$
		Since $u^*$ is not contained in any PC cycle of length $\ell+1$, by Lemma \ref{lem:not ext}, $S=S_1\cup S_2\cup S_3$. 
		\begin{claim}\label{cl:direction}
			$S=S_2$ or $S=S_3$.
		\end{claim}
		\begin{proof}
			First we prove that $S_1=\emptyset$. Suppose not. Then for each $y\in S_1$, let $f(y)$ be the unique color in $col(y, C)$. Since $G$ contains no monochromatic triangle, we have $f(y)\not\in col(C)$. For each vertex $z\in S_2$ (if $S_2\neq \emptyset$), choose $x\in V(C)$ such that $u^*\not\in \{x,x^-\}$. Since $yx^+Cx^-zy$ and $yxCx^{--}zy$ are not a PC cycles, we have 
			$$col(zy)\in \{f(y), col(zx^-)\} \cap \{f(y), col(zx^{--})\}.$$ 
			Note that $col(zx^-)\neq col(zx^{--})$ and $\{col(zx^-), col(zx^{--})\}\subseteq col(C)$. We have $col(zy)=f(y)$. Similarly, for each vertex $w\in S_3$, we have $col(wy)=f(y)$. For a vertex $y'\in S_1-y$, if $col(yy')\not\in col(\{y,y'\}, C)$, then choose a vertex $x\in V(C)$ such that $x\neq u^*$. We get a PC cycle $yx^+Cx^-y'y$ of length $\ell+1$ containing $u^*$, a contradiction. So either $col(yy')=f(y)$ or $col(yy')=f(y')$. In summary, $S_1$ is a degenerate set of $G$, a contradiction. So  $S_1=\emptyset$.
\begin{figure}[h]
	\centering
	\subfigure[$col(x^-x^+)=1$]{
		\includegraphics[width=0.22\textwidth]{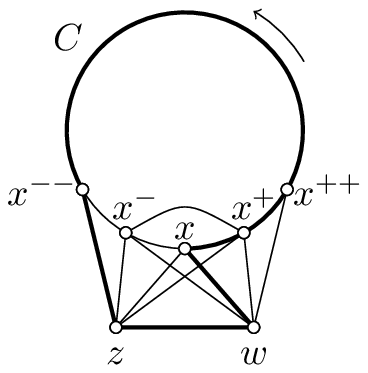}}
	\hskip 25 pt
	\subfigure[$col(x^-x^+)=2$]{
		\includegraphics[width=0.22\textwidth]{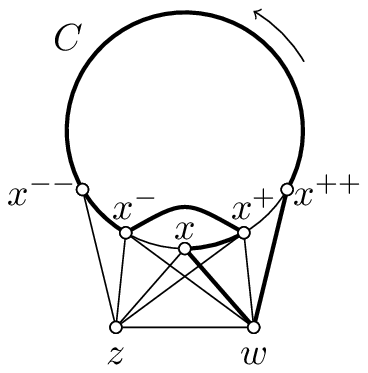}}
	\caption{Two cases when $S_1,S_2\neq \emptyset$ \label{fig:S-S2}}
\end{figure}	

			Suppose there are vertices $z\in S_2$ and $w\in S_3$. Choose a vertex $x\in V(C)$ such that $u^*\not\in \{x,x^+,x^-\}$ (this is possible since $\ell\geq 4$). Assume that $col(xx^+)=1$ and $col(xx^-)=2$. Then $col(zx^-)=2$ and $col(wx^+)=1$. 		
			Since $zwx^+Cx^-z$ is not a PC cycle, we have $col(zw)\in\{1,2\}$. Without loss of generality, assume that $col(zw)=1$. Since $x^+wx^{++}Cx^{--}zx^-x^+$ (it is possible that $x^{++}=x^{--}$) is not a PC cycle, we get $col(x^-x^+)\in\{1,2\}$. 		
			If $col(x^-x^+)=1$ (see Figure \ref{fig:S-S2}(a)), then $col(wx^-)\neq 1$ (otherwise $wx^+x^-w$ is a monochromatic triangle). This implies that $col(x^-x^{--})\neq 1$. Thus $col(zx^{--})\neq 1$. Hence $wxCx^{--}zw$ is a PC cycle of length $\ell+1$ containing $u^*$, a contradiction. If $col(x^-x^+)=2$ (see Figure \ref{fig:S-S2}(b)), then $col(zx^+)\neq 2$ (otherwise, $zx^-x^+z$ is a monochromatic triangle). This implies that $col(x^+x^{++})\neq 2$ and $col(wx^{++})\neq 2$. Hence $x^+xwx^{++}Cx^-x^+$  is a PC cycle of length $\ell+1$ containing $u^*$, a contradiction. 	
				
			In summary, either 	$S=S_2$ or $S=S_3$.
		\end{proof}
		
		Now without loss of generality, assume that $S=S_2$, $C=v_1v_2\cdots v_{\ell}v_1$ and $v_i^+=v_{i+1}$ for all $i\in [1,\ell]$ (indices are taken modulo $\ell$).  For each $i\in [1,\ell]$, define  $f(v_i)=col(v_iv_{i+1})$. Then $col(zv_i)=f(v_i)$ for all $z\in S$.
		\begin{claim}\label{cl:repeat}
			For $1\leq i<j\leq \ell$, if $col(v_iv_j)\neq f(v_i) \text{~and~} col(v_iv_j)\neq f(v_j)$, then $col(v_{i-t}v_{j-t})=f(v_{i-t-1})=f(v_{j-t-1})$ for all $t\in [1,\ell]$ (indices are taken modulo $\ell$).
		\end{claim}
		\begin{proof}
			Let $z$ be a vertex in $S$.
			Since $col(v_iv_j)\neq f(v_i) \text{~and~} col(v_iv_j)\neq f(v_j)$, the edge $v_iv_j$ is a chord of $C$ and $v_{i+1}v_iv_jv_{j+1}$ is a PC path. If $f(v_{i-1})\neq f(v_{j-1})$, then $zv_{i-1}\bar{C}v_{j}v_iC v_{j-1}z$ (See Figure \ref{fig:chord}(a)) is a PC cycle of length $\ell+1$ containing $u^*$, a contradiction. So we can assume that $f(v_{i-1})=f(v_{j-1})=\alpha$. Since $zv_{i-1}v_{j-1}z$ is not a monochromatic triangle, we have $col(v_{i-1}v_{j-1})\neq\alpha$. Consider the cycle  $zv_{j-1}v_{i-1}\bar{C} v_jv_iCv_{j-2}z$ (See Figure \ref{fig:chord}(b)). We get $col(v_{i-1}v_{j-1})=f(v_{i-2})$. Similarly, consider the cycle $zv_{i-1}v_{j-1}\bar{C} v_iv_jCv_{i-2}z$ (See Figure \ref{fig:chord}(c)). We get $col(v_{i-1}v_{j-1})=f(v_{j-2})$. Hence we have $col(v_{i-1}v_{j-1})=f(v_{i-2})=f(v_{j-2})$. Note that $f(v_{i-2})\neq f(v_{i-1})$ and $f(v_{j-2})\neq f(v_{j-1})$. Thus $col(v_{i-1}v_{j-1})\neq f(v_{i-1})$ and $col(v_{i-1}v_{j-1})\neq f(v_{j-1})$. By repeating the above argument, we finally obtain $col(v_{i-t}v_{j-t})=f(v_{i-t-1})=f(v_{j-t-1})$ for all $t\in [1,\ell]$.
		\end{proof}
		\begin{figure}[h]
			\centering
			\subfigure[$f(v_{i-1})\neq f(v_{j-1})$]{
				\includegraphics[width=0.21\textwidth]{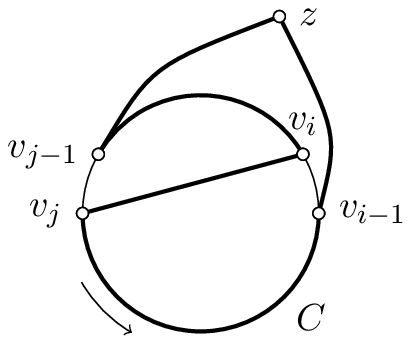}}
			\hskip 25 pt
			\subfigure[$col(v_{i-1}v_{j-1})\neq f(v_{i-2})$]{
				\includegraphics[width=0.25\textwidth]{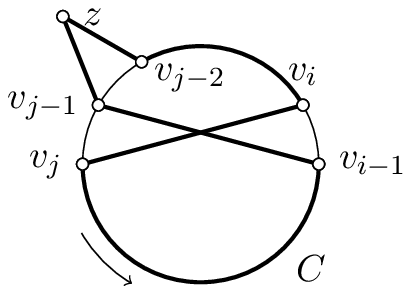}}
			\hskip 25 pt
			\subfigure[$col(v_{i-1}v_{j-1})\neq f(v_{j-2})$]{
				\includegraphics[width=0.25\textwidth]{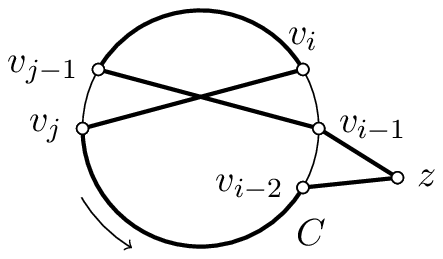}}
			\caption{Three cases in the proof of Claim \ref{cl:repeat}\label{fig:chord}}
		\end{figure}
		
		\begin{claim}\label{cl:period}
			For each $i\in [1,\ell]$, the following statements hold (indices are taken modulo $\ell$).\\
			$(a)$ $\ell$ is even and $f(v_i)=f(v_{\frac{\ell}{2}+i})$;\\
			$(b)$ $v_iv_{i+1}\cdots v_{i+\frac{\ell}{2}}$ and $v_{i+\frac{\ell}{2}}v_{i+\frac{\ell}{2}+1}\cdots v_{i+\ell}$ are rainbow paths;\\     
			$(c)$ $col(v_{i}v_{i+\frac{\ell}{2}})=f(v_{i-1})=f(v_{i+\frac{\ell}{2}-1})$.
		\end{claim}
		\begin{proof}
			If $col(v_iv_j)=f(v_i) \text{~or~} col(v_iv_j)=f(v_j)$ for all $i,j$ with $1\leq i<j\leq \ell$, then we can easily see that $V(C)$ is a degenerate set of $G$, which contradicts that $G$ is non-degenerate. So $C$ has a chord $v_iv_j$ satisfying  $col(v_iv_j)\neq f(v_i) \text{~and~} col(v_iv_j)\neq f(v_j)$ . By Claim \ref{cl:repeat}, we get $col(v_{i-t}v_{j-t})=f(v_{i-t-1})=f(v_{j-t-1})$ for all $t\in [1,\ell]$. Thus each color in $col(C)$ appears at least twice on $C$. Suppose that $\ell$ is even with $|j-i|\neq \frac{\ell}{2}$ or $\ell$ is odd. Then the color $f(v_i)$ appears at least three times on $C$. Assume that $f(v)=f(u)=f(w)=\alpha$ for three vertices $u,v,w\in V(C)$. Since $zvuz$ is not a monochromatic triangle, there holds $col(uv)\neq \alpha$. Similarly, we have $col(uw)\neq \alpha$ and $col(vw)\neq \alpha$. Applying Claim \ref{cl:repeat} to the edge $uv$, we get $col(uv)=f(u^-)=f(v^-)$ by choosing $t=\ell$. By the symmetry between $u,v$ and $w$, we finally get  
			$f(u^-)=f(v^-)=f(w^-)$ (say ``$=\beta$'') and $col(uv)=col(uw)=col(vw)=\beta$, which is a monochromatic triangle, a contradiction. Therefore, $\ell$ is even, $|j-i|=\frac{\ell}{2}$ and each color in $col(C)$ appears exactly twice on $C$. Thus $v_1v_2\cdots v_{\frac{\ell}{2}+1}$ and $v_{\frac{\ell}{2}+1}v_{\frac{\ell}{2}+2}\cdots v_\ell v_1$ are rainbow paths with  $f(v_i)=f(v_{\frac{\ell}{2}+i})$  and  $col(v_{i}v_{i+\frac{\ell}{2}})=f(v_{i-1})=f(v_{i+\frac{\ell}{2}-1})$ for all $i\in [1,\ell]$.
		\end{proof}
		
		\begin{claim}\label{cl:minus2}
			If $\ell\geq 6$, then $col(v_iv_{i-2})=f(v_i)$ for all $i\in[1,\ell]$.
		\end{claim}
		\begin{proof}
			For each $i\in [1,\ell]$, Claim \ref{cl:period} implies that the segment $v_iv_{i+1}\cdots v_{i+s}$ with $1\leq s\leq \frac{\ell}{2}$ is a rainbow path. In particular, since $\ell\geq 6$, the path $v_{i-2}v_{i-1}v_iv_{i+1}$ is rainbow. Hence $f(v_i)\neq f(v_{i-2})$. By Claim \ref{cl:repeat} (choose $j=i-2$ and $t=\ell-1$), we can see that $col(v_iv_{i-2})\in \{f(v_i), f(v_{i-2})\}$. Suppose that $col(v_iv_{i-2})\neq f(v_i)$. Then $col(v_iv_{i-2})=f(v_{i-2})$. Let $z$ be a vertex in $S$. By Claim \ref{cl:period}, we can check that $v_iCv_{i+\frac{\ell}{2}-2}zv_{i-1}v_{i+\frac{\ell}{2}-1}Cv_{i-2}v_i$ (See Figure \ref{fig:cross})
			is PC cycle of length $\ell+1$ containing $u^*$, a contradiction. 
		\end{proof}
		\begin{figure}[h]
			\centering
			\includegraphics[width=0.28\textwidth]{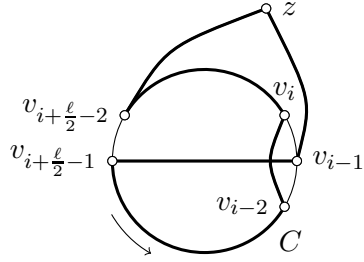}
			\hskip 25 pt
			\caption{$col(v_iv_{i-2})=f(v_{i-2})$\label{fig:cross}}
		\end{figure}
		Now by Claim \ref{cl:period}, we assume that $f(v_i)=f(v_{i+\frac{\ell}{2}})=i$ for all $i\in [1,\frac{\ell}{2}]$. 
		Note that $\ell$ is even. We will complete the proof by discussing the value of $\ell$.
		
		If $\ell=6$, then we have $col(v_1v_4)=3$ (by Claim \ref{cl:period} (c)) and $col(v_1v_3)=3$ (by Claim  \ref{cl:minus2}). Note that $col(v_3v_4)=3$. We obtain a monochromatic triangle $v_1v_3v_4v_1$, a contradiction. 
		
		If $\ell\geq 8$, then let $z$ be a vertex in $S$. Let $C_1=v_1v_3v_5\cdots v_{\ell-1}v_1$ and $C_2=v_2v_4v_6\cdots v_\ell v_2$. Then according to Claim \ref{cl:minus2}, we know that $C_1$ and $C_2$ are PC cycles. 
		If $\frac{\ell}{2}$ is odd, then $\ell\geq 10$ and $v_3v_{3+\frac{\ell}{2}}C_2v_{1+\frac{\ell}{2}}v_1\bar{C_1}v_5zv_3$ (See Figure \ref{fig:large-l}(a)) is a PC cycle of length $\ell+1$ containing $u^*$, a contradiction.     
		If $\frac{\ell}{2}$ is even, then $zv_3C_1v_{1+\frac{\ell}{2}}v_1\bar{C_1}v_{3+\frac{\ell}{2}}v_{4+\frac{\ell}{2}}C_2v_2v_{2+\frac{\ell}{2}}\bar{C_2}v_4z$ (See Figure \ref{fig:large-l}(b)) is a PC cycle of length $\ell+1$ containing $u^*$, a contradiction.
		\begin{figure}[h]
			\centering
			\subfigure[$\frac{\ell}{2}$ is odd]{
				\includegraphics[width=0.35\textwidth]{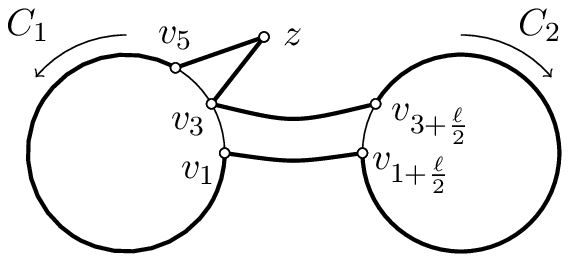}}
			\hskip 25 pt
			\subfigure[$\frac{\ell}{2}$ is even]{
				\includegraphics[width=0.42\textwidth]{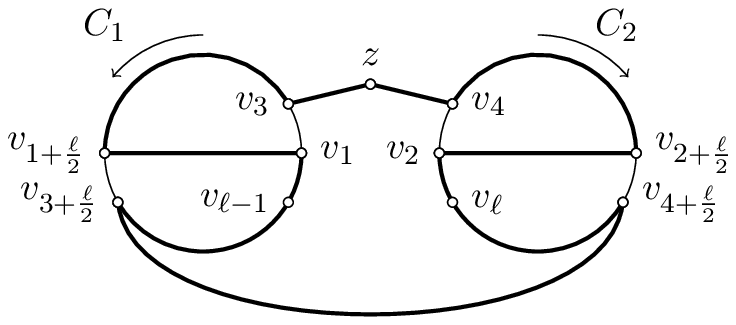}}
			\caption{Two cases when $\ell\geq 8$\label{fig:large-l}}
		\end{figure}
		
		Therefore, $\ell=4$. By Claim \ref{cl:period}, we get $f(v_1)=f(v_3)=1$, $f(v_2)=f(v_4)=2$, $col(v_1v_3)=col(v_1v_{1+\frac{\ell}{2}})=f(v_4)=2$ and $col(v_2v_4)=col(v_2v_{2+\frac{\ell}{2}})=f(v_1)=1$. If $|G|\geq 6$, then there are two distinct vertices $z,z'\in S$. 
		Since $G$ contains no monochromatic triangles. we have $col(zz')\not\in \{1,2\}$. Without loss of generality, assume that $u^*=v_1$. Then $zz'v_2v_4v_1z$ is a PC cycle of length $5$ containing $u^*$, a contradiction.
		In summary, we have $\ell=4$ and $|G|=5$. It is easy to check that $G$ is a $K_5$ containing two edge-disjoint monochromatic pentagons.
			
		The proof is complete.
	\end{proof}
	\begin{proof}[\bf Proof of Theorem \ref{thm:main}]
		Theorem \ref{thm:main} can be proved immediately by Theorems \ref{thm:deg-vertex-pan} and \ref{thm:non-deg}.
	\end{proof}

\end{document}